\documentclass{amsart}

\usepackage[latin1]{inputenc}
\usepackage{subfigure,amsmath,color}
\usepackage[comma,numbers]{natbib}
\usepackage{graphicx}
\usepackage{mathptmx}

\newtheorem{theorem}{Theorem} %[section]
\newtheorem{proposition}[theorem]{Proposition}
\newtheorem{lemma}[theorem]{Lemma}
\newtheorem{corollary}[theorem]{Corollary}
\theoremstyle{definition}
\theoremstyle{definition}\newtheorem{example}[theorem]{Example}

\newcommand{\bE}{\mathbb{E}}
\newcommand{\bF}{\mathbb{F}}

\newcommand{\bN}{\mathbb{N}}
\newcommand{\bP}{\mathbb{P}}

\newcommand{\bR}{\mathbb{R}}

\newcommand{\bS}{\mathbb{S}}

\newcommand{\cA}{\mathcal{A}}
\newcommand{\cB}{\mathcal{B}}

\newcommand{\cF}{\mathcal{F}}
\newcommand{\cG}{\mathcal{G}}
\newcommand{\cH}{\mathcal{H}}

\newcommand{\cM}{\mathcal{M}}

\newcommand{\brend}{\hfill $\triangleleft$} %\\[-2mm]}

\allowdisplaybreaks

\begin{document}

\title[]{A boundary theory approach to de Finetti's theorem}

\author{Julian Gerstenberg}
\address{Institut f\"ur Mathematische Stochastik\\ 
         Leibniz Universit\"at Hannover\\
         Postfach 6009\\
         30060 Hannover\\
         Germany}

\email{jgerst@stochastik.uni-hannover.de}

\author{Rudolf Gr\"ubel}
\address{Institut f\"ur Mathematische Stochastik\\ 
         Leibniz Universit\"at Hannover\\
         Postfach 6009\\
         30060 Hannover\\
         Germany}

\email{rgrubel@stochastik.uni-hannover.de}

\author{Klaas Hagemann}
\address{Institut f\"ur Mathematische Stochastik\\ 
         Leibniz Universit\"at Hannover\\
         Postfach 6009\\
         30060 Hannover\\
         Germany}

\email{hage@stochastik.uni-hannover.de}

\subjclass[2000]{60G09, 60J50}
\keywords{Boundary theory, exchangeability, Markov chains.}
\date{\today}

\begin{abstract} 
We show that boundary theory for transient Markov chains, as initiated by Doob, can be used to prove
de Finetti's classical representation result for exchangeable random sequences. We also 
include the relevant parts of the theory, with full proofs.
\end{abstract}

\maketitle

\section{Introduction}\label{sec:intro}
In words, de Finetti's theorem says that ``an infinite exchangeable sequence is a mixture of i.i.d.\ sequences''. 
For the history of this iconic structural result of probability theory, its many extensions and applications, and 
several different proofs, we refer the reader to the Saint Flour lecture notes~\cite{AldousSF} by Aldous 
(from which we have taken the above quotation) and to the more recent monograph~\cite{MR2161313} by Kallenberg.
Our aim in this note is to give a proof of de Finetti's theorem based on the boundary theory
for transient Markov chains as initiated by Doob~\cite{Doob59}, and to make this essentially self-contained 
by also including the relevant parts of the latter theory, again with proofs.

In order to be able to achieve this we have to leave aside many interesting and important aspects, such as
convexity considerations. In the case of de Finetti's theorem these provide a connection to ergodic theory,
which leads to a short proof based on the mean ergodic theorem~\cite[p.25f]{MR2161313}. Further, we restrict ourselves
to what we call combinatorial Markov chains. This class is sufficiently rich to provide a framework for many 
sequentially growing random discrete structures, such as permutations~\cite{VK81}, compositions and 
partitions~\cite{Gne97}, various binary tree models~\cite{EGW1,EGW2,EW16}, graphs~\cite{Grue15}, 
words~\cite{ChoiEvans}, and many others, with applications to the representation theory of the infinite
symmetric group, population genetics, and the analysis of algorithms. Both lists are far from complete. 

Among the above, \cite{Gne97} is an interesting example for the use of boundary theory in the context
of the analysis of exchangeable random structures. Recently, in the other direction, exchangeability results have
employed to determine the boundary for various combinatorial Markov chains; 
see~\cite{ChoiEvans, EGW2, EW16}. 

In the next section we introduce some notation and then give a formal statement of de Finetti's result
in Theorem\ \ref{thm:deFinetti}.
This is followed by an exposition of the boundary theory that we need. 
In Section~\ref{sec:dFproof} we then  prove Theorem~\ref{thm:deFinetti} in two
steps: We use the material from Section~\ref{sec:boundaryTheory} to obtain the result
for finite state spaces, which turns out to be very easy once the boundary theory is in place, and then apply standard 
arguments for the lift to more general spaces.

%In the final Section~\ref{sec:disc} we discuss the 
%link between exchangeable structures and Markov chain boundaries in more detail. 

\section{de Finetti's theorem}%
\label{sec:dFthm}
We need some definitions. The random variables $X_n$, $n\in\bN$, are all defined on a common
probability space $(\Omega,\cA,\bP)$ and take their values in a measurable space $(E,\cB)$ which
we assume to be Borel, meaning that there exists a measure-theoretic isomorphism between 
this space and a Borel subset of the real line.
We write $(E_\infty,\cB_\infty)$ for the corresponding path space, where $E_\infty$ is the 
set of all sequences $x=(x_n)_{n\in\bN}$ of elements $x_n$ of $E$ and $\cB_\infty$ is generated 
by the projections $x=(x_n)_{n\in\bN}\mapsto x_i$ from $E_\infty$ to $E$, $i\in\bN$. We further write $\cM_1(E)$ 
and $\cM_1(E_\infty)$ for the set of probability measures on $(E,\cB)$ and $(E_\infty,\cB_\infty)$
respectively. 
In this setup we may regard the sequence $X$ as a random variable (measurable mapping)  with values in 
$(E_\infty,\cB_\infty)$ so that its distribution, the push-forward $\bP^X$ of $\bP$ under
the mapping $X$, is an element of $\cM_1(E_\infty)$.

Let $\bS$ be the set of finite permutations $\pi$ of $\bN$, meaning that $\pi:\bN\to \bN$ 
is bijective and that $\#\{i\in\bN:\, \pi(i)\not= i\}<\infty$. Each such permutation defines 
a function $\phi_\pi:E_\infty \to E_\infty$, $E_\infty \ni x=(x_n)_{n\in\bN}\mapsto (x_{\pi(n)})_{n\in \bN}$
which, again via push-forward, induces a mapping from and to $\cM_1(E_\infty)$. We say that
$X$ is exchangeable if its distribution is invariant under all $\phi_\pi$, $\pi\in \bS$. 

Spaces of measures such as $\cM_1(E)$ can be endowed with a measurable structure by 
requiring that $\mu\mapsto \mu(A)$ be measurable for all measurable sets~$A$; we write 
$\cB(\cM_1(E))$ for the resulting $\sigma$-field. In particular, it then makes sense to regard a 
random measure as a random variable with values in $\cM_1(E)$.  

\begin{theorem}[de Finetti]\label{thm:deFinetti}
If $X=(X_n)_{n\in\bN}$ is exchangeable then there exists a random variable $M$ with values in $\cM_1(E)$ 
such that, conditionally on $M=\mu$,  the $X_n$'s are i.i.d.\ with common distribution $M=\mu$.   
\end{theorem}

The conditioning means that, for all $k\in\bN$ and $A_1,\ldots,A_k\in \cB$,
\begin{equation}\label{eq:cond}
  \bP(X_1\in A_1,\ldots, X_k\in A_k) \,=\, \int\, \prod_{i=1}^k \mu(A_i) \ \bP^M(d\mu).
\end{equation}
Again, $\bP^M$ denotes a push-forward, now the distribution of the random measure $M$,
which is a probability measure on $(\cM_1(E),\cB(\cM_1(E)))$.
The interpretation is that
of a two-stage procedure: First, we select $M$, the \emph{directing measure}, according to $\bP^M$.
Then, if this selection resulted in $\mu\in\cM_1(E)$, we generate the $X_n$'s independently, all 
with distribution~$\mu$. Alternatively, using the concept of conditional probability,
we first define a probability kernel $Q$ from $(\cM_1(E),\cB(\cM_1(E))$ to $(E_\infty,\cB_\infty)$ by
$Q(\mu,.)=\mu^{\otimes \bN}$, where `$\otimes$' denotes product measure. The distribution of $X$
is then the product of the distribution of the directing measure and this kernel in the sense that
\begin{equation*}
  \bP(X\in A) = \int Q(\mu,A)\, \bP^{M}(d\mu)\quad\text{for all } A\in\cB_\infty.
\end{equation*}
This is, of course, the same as \eqref{eq:cond} and may be written shortly as $Q=\bP^{X|M}$.  
Both interpretations also arise in connection with boundaries.

\section{Boundary theory for combinatorial Markov chains}%
\label{sec:boundaryTheory}
By a combinatorial Markov chain (CMC) on $\bF$ we mean a Markov chain $Y=(Y_n)_{n\in\bN_0}$ with 
values in a combinatorial family $\bF$ that is adapted to $\bF$ in the sense that
\begin{equation*}
  \bP(Y_n\in\bF_n)=1 \ \text{ for all } n\in\bN_0.
\end{equation*}
Here $\bF_n$ is the (finite) set of objects $x$ with `size' $n$, and $\bF=\bigcup_{n=0}^\infty\bF_n$.
This implies that  time $n$ is a function of state $x$. We further assume that there is only one object $e$ 
of size 0, i.e.\ $\bF_0=\{e\}$ and $Y_0\equiv e$, and that the chain is weakly irreducible in the sense that
$\bP(Y_n=x)>0$ for all $n\in\bN$, $x\in\bF_n$. 

In general, a stochastic process such as $Y$ is based
on an underlying probability space $(\Omega,\cA,\bP)$ which is often left unspecified. In the
present context it is no loss of generality to work with the canonical construction, which here means that we 
take $\Omega$ to be the path space, i.e.\ $\Omega= \bF_\infty:=\prod_{n=0}^\infty \bF_n$;
we define $Y_n:\bF_\infty\to\bF_n$ as the projection to the
$n$th sequence element; and we let $\cA$ be the $\sigma$-field generated by the $Y_n$'s,  
which we denote by~$\cB_\infty$. Further, for all $n\in\bN_0$, let $\cF_n:=\sigma(\{Y_m:\, m\le n\})$
and $\cG_n:=\sigma(\{Y_m:\, m\ge n\})$.

For $m\le n$, $x\in\bF_m$, $y\in\bF_n$ let
\begin{equation} \label{eq:MartinK}
  K(x,y) := \frac{\bP(Y_n=y|Y_m=x)}{\bP(Y_n=y)}\quad 
                       \biggl(\  =\frac{\bP(Y_m=x|Y_n=y)}{\bP(Y_m=x)}\le \frac{1}{\bP(Y_m=x)}\Biggr),
\end{equation}
and put $K(x,y)=0$ otherwise. 
This is the \emph{Martin kernel}. Consider the discrete topology on $\bF$. Then 
\begin{equation*}
  \bF\ni y \ \mapsto\ \bigl(x\mapsto K(x,y)\bigr) 
               \in \prod_{m=1}^\infty\prod_{x\in\bF_m} \Bigl[0\, ,\, \frac{1}{\bP(Y_m=x)}\Bigr]
\end{equation*}
provides an embedding of $\bF$ into a compact space. 
We ignore the distinction between $\bF$ and its image in the product space. Then
the \emph{Doob-Martin compactification} $\bar\bF$ is the closure of $\bF$ in this
space, and $\partial\bF:=\bar\bF\setminus \bF$ is the \emph{Martin boundary}. 
By construction, $(y_n)_{n\in\bN}$ converges in the Doob-Martin topology if and only if the sequences
$\bigl(K(x,y_n)\bigr)_{n\in\bN}$ converge in $\bR$ for all $x\in\bF$. In the present context this implies
that either $y_n=y$ for all but finitely many $n\in\bN$ with some $y\in\bF$, or that $m_n\to\infty$ 
with $m_n$ given by $y_n\in\bF_{m_n}$.

Probabilistically we may interpret this as follows: Suppose that $(y_n)_{n\in\bN}$
is such that $y_n\in\bF_{m_n}$ with $m_n\to\infty$.
Then the sequence converges if and only if, for all fixed $k\in\bN$ 
and $x_0\in\bF_0,\ldots,x_k\in\bF_k$, $ \bP(Y_0=x_0,\ldots,Y_k=x_k| Y_{m_n}=y_n)$
converges as $n\to\infty$. This may be rephrased as
\begin{equation*}
  y_n \text{ converges }\  \Longleftrightarrow \ \; \bP^{\,Y|Y_{m_n}=y_n}\,\text{ converges}
\end{equation*}
if we interpret process convergence on the right hand side 
as convergence of the respective finite-dimensional distributions.

The following limit theorem is the first main result of the theory.

\begin{theorem}\label{thm:limit}
  In $\bar\bF$, $Y_n$ converges almost surely to some $Y_\infty$ with values in $\partial\bF$.
\end{theorem}

\begin{proof}
We need to show that, for every fixed $x\in\bF$, $K(x,Y_n)$ converges almost surely as $n\to\infty$.
With $x\in\bF_k$ and $n> k$ it is straightforward to show that $\bE\bigl[K(x,Y_n)\big| \cG_{n+1}\bigr]=K(x,Y_{n+1})$.
Now use the backwards martingale convergence theorem.
\end{proof}

In the CMC situation, $h:\bF \to \bR$ is \emph{harmonic} if $h(x) = \sum_{y\in\bF_{n+1}} \bP(Y_{n+1}=y|Y_n=x)\, h(y)$
for all $n\in\bN_0$, $x\in\bF_n$.  
Equivalently,
$(h(Y_n),\cF_n)_{n\in\bN_0}$ is a (forward) martingale. 
We write $\cH_{1,\ge}$ and $\cH_{1, >}$ for the set of all non-negative respectively positive harmonic functions $h$ 
that are normed in the sense that $h(e)=1$.  As a consequence of the fact that we only assume weak irreducibility
these sets may well be different; let $\bF_h:=\{x\in\bF:\, h(x)>0\}$.
Each $h\in\cH_{1,\ge}$ defines a  transition mechanism $p_h$ on $\bF_h$ via 
\begin{equation}\label{eq:transh}
  p_h(x,y) \, :=\,  \frac{1}{h(x)} \bP(Y_{n+1}=y|Y_n=x)\, h(y), 
\end{equation}
with $n\in\bN$, $x\in\bF_n\cap\bF_h$, $y\in\bF_{n+1}\cap\bF_h$.
The corresponding probability measure on the path space
is the \emph{$h$-transform} $\bP_h$ of~$\bP$. (The transform is often regarded as acting on the 
process $Y$. Here $Y$ is fixed as as we work with the canonical construction; we will occasionally say that $Y$ is a
CMC on $\bF_h$ under $\bP_h$.) Then 
\begin{equation*}
  \cM(\bP) := \{\bP_h:\, h\in \cH_{1,\ge}\}
\end{equation*}
is a family of probability measures on $(\bF_\infty,\cB_\infty)$, and it is
easy to check that 
\begin{equation}\label{eq:dens}
  \frac{\bP_h(Y_1=x_1,\ldots,Y_n=x_n)}{\bP(Y_1=x_1,\ldots,Y_n=x_n)} = h(x_n)
\end{equation}
whenever $\bP(Y_1=x_1,\ldots,Y_n=x_n)>0$. In words: $x \mapsto h(x)$, $x\in\bF_n$, is the density
of $\bP_h$ with respect to $\bP$, if both are restricted to $\cF_n$. This implies that
the Martin kernel $K_h$ associated with $\bP_h$, $h\in\cH_{1,>}$, is related to the original
Martin kernel $K$ by
\begin{equation}\label{eq:kernel_of_h-transform}
  K_h(x,y) \, = \, \frac{1}{h(x)} \, K(x,y) \ \text{ for all } x,y\in\bF.
\end{equation}
From this it is immediate that all $h\in\cH_{1,>}$ lead to the same Doob-Martin compactification and,
consequently, to  the same Martin  boundary. For $h\in\cH_{1,\ge}\setminus\cH_{1,>}$ the shrinkage from 
$\bF$ to $\bF_h$ may lead to a smaller boundary $\partial\bF_h$, but we may always regard $\bP_h$ as
(the distribution of) a CMC on $\bF_h$. Further,
the canonical embedding of $\bF_h$ into $\bF$ extends uniquely to a continuous and injective 
function $\phi_h: \bar\bF_h\to \bar\bF$. In particular, $Y_n$ converges almost surely to $Y_\infty$ as $n\to\infty$ 
with respect to $\bP_h$, for all $h\in\cH_{1,\ge}$. 

By construction, each function $K(x,\cdot)$, $x\in\bF$, has a unique continuous extension from $\bF$ to $\bar\bF$.
It is easy to see that
each $K(\cdot,\alpha)$, $\alpha\in \partial\bF$, is harmonic and, in fact, an element of $\cH_{1,\ge}$. For
$h\in\cH_{1,\ge}$, the corresponding $h$-transform  $\bP_h$ similar leads to a continuous extension of
the corresponding Martin kernel $K_h$, and the embedding $\phi_h$ gives the following extension 
of~\eqref{eq:kernel_of_h-transform},
\begin{equation}\label{eq:kernel_of_h-transform2}
  K_h(x,\alpha) = \frac{1}{h(x)} K(x,\alpha)\quad\text{for all } x\in\bF_h, \, \alpha\in\partial\bF.
\end{equation}
We can now state and prove the second main result.

\begin{theorem}\label{thm:repr}
Suppose that $h\in\cH_{1,\ge}$. Then 
\begin{equation}\label{eq:repr_h1}
  h(x) \; =\; \int K(x,\alpha)\, \bP_h^{Y_\infty}(d\alpha) \quad\text{for all } x\in\bF_h.
\end{equation}
\end{theorem}

\begin{proof} 
From Theorem~\ref{thm:limit} and the definition of the Doob-Martin topology we know that $K(x,Y_n)$
converges to $K(x,Y_\infty)$ $\bP$-almost surely for all $x\in\bF$, and we further know from its proof that
$(K(x,Y_n),\cG_n)_{n\ge m}$ is a backwards martingale if $x\in\bF_m$. For these $L^1$-convergence is automatic, so that
\begin{equation*}
  \lim_{n\to\infty} \bE K(x,Y_n) = \bE K(x,Y_\infty) = \int K(x,\alpha)\, \bP^{Y_\infty}(d \alpha).
\end{equation*}
Using the second expression for $K$ in~\eqref{eq:MartinK} we get, for $n>m$, 
\begin{equation*}
  \bE K(x,Y_n) \;=\; \sum_{y\in\bF_n} \frac{\bP(Y_m=x|Y_n=y)}{\bP(Y_m=x)}\, \bP(Y_n=y)\; = \; 1. %\qedhere
\end{equation*}
Taken together this shows that 
\begin{equation}\label{eq:repr_h1}
  \int K(x,\alpha)\, \bP^{Y_\infty}(d\alpha) = 1 \quad\text{for all } x\in\bF.
\end{equation}
Now fix $h\in\cH_{1,\ge}$. Repeating the above for the CMC on $\bF_h$ given by the corresponding $h$-transform $\bP_h$ with
Martin kernel $K_h$  we obtain
\begin{equation}\label{eq:repr_h1}
  \int K_h(x,\alpha)\, \bP_h^{Y_\infty}(d\alpha) = 1 \quad\text{for all } x\in\bF_h.
\end{equation}
An appeal to \eqref{eq:kernel_of_h-transform2} now concludes the proof.
\end{proof}

The above is the classical representation result of the theory, with additional information about the
representing measure. Again, we may rephrase this in more probabilistic terms. Let $Q:\partial\bF\times \cB_\infty\to [0,1]$ 
be defined by
\begin{equation*}
  Q(\alpha,A) = \bP_{K(\cdot,\alpha)}(A)\quad\text{for all } \alpha\in\partial \bF, \, A\in \cB_\infty.
\end{equation*}
Then $A\to Q(\alpha,A)$ is a probability measure on $(\bF_\infty,\cB_\infty)$ for all 
$\alpha\in\partial\bF$, and it is easy to show that $\alpha\mapsto Q(\alpha,A)$
is $\cB(\partial\bF)$-measurable for all $A\in\cB_\infty$, where $\cB(\partial\bF)$ denotes
the Borel $\sigma$-field on $\partial\bF$. Hence $Q$ is a probability kernel from
$(\partial \bF,\cB(\partial\bF))$ to $(\bF_\infty,\cB_\infty)$. 

\begin{corollary}\label{cor:condProb}
The kernel $Q$ is a version of $\,\bP_h^{Y|Y_\infty}$ for all $h\in\cH_{1,\ge }$.  
\end{corollary}

\begin{proof} We need to show that, for all $A\in\cB_\infty$,
\begin{equation*}
  \bP_h(A) = \int Q(\alpha,A)\, \bP_h^{Y_\infty}(d\alpha).
\end{equation*}
Because of the structure of $\cB_\infty$ is is enough to do so for sets $A$ of the form
\begin{equation*}
A=\bigl\{(y_n)_{n\in\bN} \in \bF_\infty:\, y_1=x_1,\ldots,y_m=x_m\bigr\}, 
\end{equation*}
with $m\in\bN$ and $x_1\in\bF_1,\ldots,x_m\in\bF_m$ such that $\bP_h(A)>0$.
For such a set  $A$ we obtain, using~\eqref{eq:dens} and the definition of $Q$,
\begin{equation*}
  \bP_h(A) = h(x_m)\, \bP(A), \quad Q(\alpha,A) = K(x_m,\alpha)\, \bP(A),
\end{equation*}
so that it remains to show that 
\begin{equation}\label{eq:repr_h}
  \int K(x,\alpha)\, \bP_h^{Y_\infty}(d\alpha) = h(x) \quad\text{for all } x\in\bF_h,
\end{equation}
but this is the content of Theorem~\ref{thm:repr}.
\end{proof}

Much the same as in the de Finetti situation, see the end of Section~\ref{sec:dFthm}, a CMC may thus be regarded as
a two-stage experiment, with the boundary element $\alpha$ chosen at random and then running the 
chain according to $\bP_h$ with $h=K(\cdot,\alpha)$.

This is an elegant theory, and, in view of their space-time property, quite accessible for CMCs.
However,  to work out things for a given chain, i.e.\ to
\begin{itemize}
\item[$-$] describe $\bar \bF$ as a topological space in terms 
                                      of more familiar objects,
\item[$-$] find the distribution of $Y_\infty$,
\item[$-$] describe the $h$-transforms for $h=K(\cdot,\alpha)$, with 
                                      $\alpha \in \partial\bF$,
\end{itemize}
may be difficult. An occasionally efficient strategy is to show that the chain of interest is an $h$-transform 
of some other chain where these tasks are easy or have already been carried out. Any characterization of 
the family $\cM(\bP)$ could be useful in this context; we provide one  that will turn out to work in the de Finetti
case.

One often uses 
\begin{equation}\label{eq:MEsimple}
  \bP(Y_{n+1}=x_{n+1}|Y_n=x_n,\ldots,Y_1=x_1)\, = \, \bP(Y_{n+1}=x_{n+1}|Y_n\!=\!x_n)
\end{equation}
as a definition of the Markov property. Equivalently, for all $n\in\bN$, $\cF_n$ and $\cG_n$ are 
conditionally independent, given $\sigma(Y_n)$. The latter definition exhibits the underlying 
symmetry which can also be seen in~\eqref{eq:MartinK}. In particular we can 
build the chain backwards from any finite instant of time, using the \emph{cotransitions}
$\bF_{n+1}\times\bF_n \ni (y,x) \ \mapsto\ \bP(Y_n=x|Y_{n+1}=y)$. 
Indeed, such a time reversal has appeared repeatedly, if implicitly,  above.

\begin{proposition}\label{prop:cotrans}
\emph{(a)} For all $h\in\cH_{1,>}$, $\bP_h$ has the same cotransitions as $\bP$.

\smallbreak
\emph{(b)} If $\,\bP^\circ$ is a probability measure on $(\bF_\infty,\cB_\infty)$ such that $Y$ is  
a CMC on $\bF$ under $\bP^\circ$, and if $\,\bP^\circ$ has the same cotransitions as $\bP$, then $\bP^\circ=\bP_h\,$ 
for some $h\in\cH_{1,>}$.
\end{proposition}

\begin{proof}
Part (a) is immediate from~\eqref{eq:transh} and~\eqref{eq:dens}. For~(b) we put
$h(x)=\bP^\circ(Y_n=x)/\bP(Y_n=x)$ for $x\in\bF_n$. 
\end{proof}

We mention in passing that, once again, for CMCs the proof is very simple in view of their
space-time property.  

\begin{example}\label{ex:NERW}
We consider the family $\bF$ of weak compositions of natural numbers with $d$ parts, 
$d\in\bN$ fixed, where
\begin{equation*}
  \bF_n = \bigl\{ (x_1,\ldots,x_d)\in\bN_0^d:\, x_1+\cdots + x_d=n\bigr\} 
\end{equation*}
for all $n\in\bN$, and $\bF_0=\{(0,\ldots,0)\}$. 
For $j=1,\ldots,d$ let $e_j$ be the element of $\bF_1$ that has $x_j=1$.
Then a standard CMC on $\bF$ arises as the sequence
of partial sums for a sequence of independent random vectors $X_n$, $n\in\bN$, with $X_n$ 
uniformly distributed on $\{e_j:\, j=1,\ldots,d\}$, meaning that $Y_n=\sum_{i=1}^n X_i$ for all 
$n\in\bN_0$.  Clearly, for $x=(x_1,\ldots,x_d)\in \bF_m$ 
and $y=(y_1,\ldots,y_d)\in \bF_n$, $n>m$, with $y_i\ge x_i$ for  $i=1,\ldots,d$,
\begin{equation*}\displaystyle
  K(x,y) \, = \, \frac{\displaystyle\binom{n-m}{y_1-x_1,\ldots,y_d-x_d} \, d^{-(n-m)}}
                                  {\displaystyle\binom{n}{y_1,\ldots,y_d} \, d^{-n}}.
\end{equation*}
A special case of this is $K(e_j,y)=d y_j/n$. It follows easily that a sequence $(y_n)_{n\in\bN}$, 
$y_n=(y_{n1},\ldots,y_{nd})\in\bF_{m_n}$ with $m_n\to\infty$, converges in the Doob-Martin topology if and only if the
sequences $(y_{ni}/m_n)_{n\in\bN}$, $i=1,\ldots,d$, converge in $\bR$, and that $\partial\bF$ is
(homeomorphic to) the $(d-1)$-dimensional probability simplex 
\begin{equation*}
  S(d) := \Bigl\{ \alpha=(\alpha_1,\ldots,\alpha_d)\in \bR_+^d:\, \sum_{i=1}^d \alpha_i =1 \Bigr\}.
\end{equation*}
Further, $K(x,\alpha) = \prod_{i=1}^d \alpha_i^{x_i}$   
for all $x=(x_1,\ldots,x_d)\in\bF$, $\alpha=(\alpha_1,\ldots,\alpha_d)\in S(d)$,
so that 
\begin{equation*}
  \bP_{K(\cdot,\alpha)}(Y_{n+1}=x+e_j|Y_n=y)\, = \, \alpha_j, \quad j=1,\ldots,d.
\end{equation*}
In words: Conditioned on $Y_\infty=\alpha$, the process is a random walk with step distribution $\alpha$.
An extreme case of state space reduction arises if $\alpha=(1,0,\ldots,0)$, for example: The 
corresponding $\bF_h$ consists of the compositions $(n,0,0,\ldots,0)$, and $\bP_h$ is concentrated on a
single path.
Finally, a straightforward calculation shows that the cotransitions are given by
\begin{equation*}
  \bP(Y_n=y|Y_{n+1}=y+e_j) \, = \, \frac{y_j+1}{n+1}, \quad j=1,\ldots,d,
\end{equation*}
for all $y=(y_1,\ldots,y_d)\in \bF$. 
\brend
\end{example}

\section{The proof}%
\label{sec:dFproof}

In the first step of the proof we assume that $E$ is finite, with all subsets being measurable.
Let $d:=\#E$; 
for (notational) simplicity we may assume that $E=\{1,\ldots,d\}$.
Given the exchangeable process $X$ with values in $E$ we define a process $Y=(Y_n)_{n\in\bN_0}$ with values
in the combinatorial family $\bF$ of weak compositions with $d$ parts by
\begin{equation*}
  Y_n := \bigl(\#\{1\le i\le n:\, X_i=k\}\bigr)_{k=1,\ldots,d} \quad \text{for all } n\in\bN_0.
\end{equation*}
In words: We count how often the different values have appeared up to time $n$.
In the following lemma we collect two simple, but decisive observations.

\begin{lemma}\label{lem:YisMarkov}
The process $Y$ is a Markov chain, and its cotransitions are given by 
\begin{equation*}
  \bP(Y_n=y|Y_{n+1}=y+e_j) \, = \, \frac{y_j+1}{n+1}, \quad j=1,\ldots,d,\; y\in\bF_n.
\end{equation*}
\end{lemma}

\begin{proof} 
Using the definition \eqref{eq:MEsimple} of the Markov property we want to show that
\begin{equation*}
  \bP(Y_{n+1}=y_{n+1}|Y_n=y_n,\ldots,Y_1=y_1)
         \, =\, \bP(Y_{n+1}=y_{n+1}|Y_n=y_n)
\end{equation*}
whenever $\bP(Y_{n+1}=y_{n+1},Y_n=y_n,\ldots,Y_1=y_1)>0$.
Fix $n\in\bN$, $y_n\in\bF_n$ and $y_{n+1}\in\bF_{n+1}$. We may assume that $y_{n+1}=y_n+e_j$ for 
some $j\in\{1,\ldots,d\}$, where $e_1,\ldots,e_d$ denote the unit vectors in $\bN_0^d$, 
as in Example~\ref{ex:NERW}. Let 
\begin{equation*}
  A(y_n):= \bigl\{(y_1,\ldots,y_{n-1}):\, \bP(Y_1=y_1,\ldots,Y_n=y_n,Y_{n+1}=y_{n+1})>0\bigr\}.
\end{equation*}
As the $Y$-sequence is the sequence of partial sums of unit vectors in $\bN_0^n$
there is a bijective correspondence between $A(y_n)$ and a
subset $B(y_n)$ of $\{e_1,\ldots,e_d\}^n$, where $B(y_n)$ has the property that, for all permutations $\pi$
of $\{1,\ldots,n\}$, 
\begin{equation*}
  (x_1,\ldots,x_n)\in B(y_n) \quad\Longleftrightarrow \quad (x_{\pi(1),\ldots,\pi(n)})\in B(y_n).
\end{equation*}
The function $(x_1,\ldots,x_n)\mapsto \bP(X_1=x_1,\ldots,X_n=x_n,X_{n+1}=e_j)$
is constant on $B(y_n)$ in view of the exchangeability of $X$, which implies that
\begin{equation*}
  (y_1,\ldots,y_{n-1}) \, \mapsto \, \bP(Y_1=y_1,\ldots,Y_n=y_n,Y_{n+1}=y_{n+1})
\end{equation*}
is constant on $A(y_n)$. The same reasoning shows that 
\begin{equation*}
  (y_1,\ldots,y_{n-1}) \mapsto \bP(Y_1=y_,\ldots,Y_n=y_n)
\end{equation*}
is also constant on $A(y_n)$.
Hence, for any fixed $(y_1,\ldots,y_{n-1})\in A(y_n)$,
\begin{align*}
  \bP(Y_{n+1}=y_{n+1}|Y_n=y_n)\ 
    &=\ \frac{\bP(Y_{n+1}=y_{n+1},Y_n=y_n)}{\bP(Y_n=y_n)}\\
    &=\ \frac{\sum_{(y_1',\ldots,y_{n-1}')\in A(y_n)}
                              \bP(Y_{n+1}=y_{n+1},Y_n=y_n, Y_1=y_1',\ldots,Y_{n-1}=y_{n-1}')}
                      {\sum_{(y_1',\ldots,y_{n-1}')\in A(y_n)}
                                  \bP(Y_n=y_n ,Y_1=y_1',\ldots,Y_{n-1}=y_{n-1}')}\\
    &=\ \frac{\# A(y_n)\, \bP(Y_{n+1}=y_{n+1},Y_n=y_n, Y_1=y_1,\ldots,Y_{n-1}=y_{n-1})}
                      {\# A(y_n)\, \bP(Y_n=y_n ,Y_1=y_1,\ldots,Y_{n-1}=y_{n-1})}\\
    &=\ \bP(Y_{n+1}=y_{n+1}|Y_n=y_n,\ldots,Y_1=y_1),
\end{align*}
which is the desired equality.

Similar arguments can be used to find the cotransitions. The two events in question may be decomposed as
\begin{equation*}
  \bP(Y_n=y_n,Y_{n+1}=y_n+e_j)
     \ =\ \sum_{(x_1,\ldots,x_n)\in B(y_n)}\bP(X_1=x_1,\ldots,X_n=x_n,X_{n+1}=e_j),
\end{equation*}

\vspace{-3mm}
\noindent
and
\begin{equation*}
  \bP(Y_{n+1}=y_n+e_j)
     \ =\ \sum_{(x_1,\ldots,x_n,x_{n+1})\in B(y_n+e_j)}\bP(X_1=x_1,\ldots,X_n=x_n,X_{n+1}=x_{n+1})
\end{equation*}
respectively, with $B(y_n+e_j)$ the set of possible increments leading to $Y_{n+1}=y_n+e_j$.
Suppose that $y_{n,j}=k$. Then there are $k$ steps in direction $j$ which have to be distributed to $n$ 
positions in the first event and $(k+1)$ such steps to $n+1$ possibilities for the second. Exchangeability 
implies that the respective probabilities are the same, so that 
\begin{equation*}
  \bP(Y_n=y_n|Y_{n+1}=y_n+e_j) \, 
                     =\, \frac{%\displaystyle
                                      \binom{n}{k}}{%\displaystyle
                        \binom{n+1}{k+1}}
                  \, = \, \frac{y_{n,j}+1}{n+1}. \qedhere    
\end{equation*}
\end{proof}

For finite state space $E$, where we may assume that $P(X_n=x)>0$ for all $x\in E$, 
Theorem~\ref{thm:deFinetti} is now immediate: From Proposition~\ref{prop:cotrans}  and Example~\ref{ex:NERW} 
we know the boundary for the $Y$-process and the conditional distribution of $Y$ given $Y_\infty$. In particular,  
given $Y_\infty=\mu$, the increments $(X_n)_{n\in\bN}$ are i.i.d.\ with distribution $\mu$. Hence 
$Y_\infty$, which is a random element of the probability simplex on $E$, is the directing measure for $X$.

Suppose now that the state space for the exchangeable sequence $X=(X_n)_{n\in\bN}$ is (a Borel subset of)
the unit interval
$[0,1)$, endowed with its Borel $\sigma$-field $\cB_{[0,1)}$. There are various possibilities for the step from
finite sets to this state space. We choose one that is based on binary expansion: Let 
$\Psi=(\Psi_k)_{k\in\bN}: [0,1)\to \{0,1\}^\infty$ be defined by
\begin{equation*}
  \Psi_k(x) := \lfloor 2^k x\rfloor - 2 \lfloor 2^{k-1} x\rfloor.
\end{equation*}
On the space $\{0,1\}^\infty$ of 0-1 sequences we consider the $\sigma$-field $\cB_{\{0,1\}^\infty}$ generated
by the projections or, equivalently, by the sets
\begin{equation*}
  A(x_1,\ldots,x_k) := \bigl\{y=(y_n)_{n\in \bN}\in \{0,1\}^\infty:\, y_1=x_1,\ldots,y_k=x_k\bigr\},
\end{equation*}
with $k \in \bN$ and $x_1,\ldots,x_k\in\{0,1\}$. Then, for each fixed $k\in\bN$, we obtain an exchangeable 
sequence $Y_k=(Y_{k,n})_{n\in\bN}$ with values in the finite set $\{0,1\}^k$ via
\begin{equation*}
  Y_{k,n} := \bigl(\Psi_1(X_n),\ldots,\Psi_k(X_n)\bigr), \ n\in \bN.
\end{equation*}
Let $M_k$ be the corresponding directing measure, which is a random element of the probability simplex 
on $\{0,1\}^k$. These are obviously consistent in the sense that, with probability~1,
$M_k$ is the push-forward of $M_{k+1}$ under the projection 
\begin{equation*}
 (x_1,\ldots,x_{k-1},x_k,x_{k+1})\mapsto (x_1,\ldots,x_{k-1},x_k) 
\end{equation*}
from $\{0,1\}^{k+1}$ to $\{0,1\}^k$. Outside some null set $N$ we may therefore apply Kolmogorov's
extension theorem to obtain a probability measure $M$ on $\cB_{\{0,1\}^\infty}$ which has the $M_k$'s as its 
projections to the first $k$ coordinates, $k\in\bN$.  The 
directing measure for $X$ now results as the push-forward of $M$
under the mapping $(x_n)_{n\in\bN} \mapsto \sum_{n=1}^\infty x_n 2^{-n}$.

For the step from state spaces that are Borel subsets of the unit interval to 
general Borel spaces we refer to~\cite[Section 7]{AldousSF}.

{\small
\bibliographystyle{amsalpha}

\def\cprime{$'$}
\providecommand{\bysame}{\leavevmode\hbox to3em{\hrulefill}\thinspace}
\providecommand{\MR}{\relax\ifhmode\unskip\space\fi MR }
\providecommand{\MRhref}[2]{%
  \href{http://www.ams.org/mathscinet-getitem?mr=#1}{#2}
}
\providecommand{\href}[2]{#2}

}

\end{document}